\documentclass[final,leqno,onefignum,onetabnum]{siamltex1213}
\usepackage{latexsym,amsmath,amssymb,graphicx,slashbox}

\newtheorem{prop}{Proposition}[section]
\newtheorem{defn}{Definition}[section]
\newtheorem{obs}{Observation}
\newtheorem{cor}{Corollary}[section]
\newtheorem{rem}{Remark}[section]

\title{On weakly irreducible nonnegative tensors and  interval hull of some classes of tensors}

\author{M. Rajesh Kannan\thanks{Department of Mathematics, Technion- Israel Institute of Technology, Haifa 32000, Israel(rajeshkannan1.m@gmail.com, rajesh@tx.technion.ac.il).}
\and Naomi Shaked-Monderer\thanks{Department of Economics and Management, The Max Stern Academic College of Yezreel Valley, Yezreel Valley 19300, Israel(nomi@tx.technion.ac.il).}
\and  Abraham Berman\thanks{Department of Mathematics, Technion- Israel Institute of Technology, Haifa 32000, Israel(berman@tx.technion.ac.il).}}

\begin{document}
\maketitle
\slugger{mms}{xxxx}{xx}{x}{x--x}

\begin{abstract}
In this article we prove the strict monotonicity of the spectral radius of weakly irreducible nonnegative tensors.  As an application, we give a necessary and sufficient condition for an  interval hull of tensors to be contained in the set of all strong $\mathcal{M}$-tensors. We also establish some properties of $\mathcal{M}$-tensors. Finally, we consider some problems  related to interval hull of positive (semi)definite tensors and  $P(P_0)$-tensors.

\end{abstract}

\begin{keywords}Weakly irreducible nonnegative tensors, $\mathcal{M}$-tensors, Positive definite tensors, $P$-tensors, Interval hull of tensors.\end{keywords}

\begin{AMS}15A69\end{AMS}

\pagestyle{myheadings}
\thispagestyle{plain}
\markboth{On weakly irreducible nonnegative tensors}{M. Rajesh Kannan N. Shaked-Monderer and A. Berman}

\section{Introduction}

An $m$-order $n$-dimensional square real tensor is a multidimensional array of $n^m$ elements of the form
\begin{center}
$\mathcal{A} = (A_{i_1\dots i_m})$, $A_{i_1\dots i_m} \in \mathbb{R}$,  $1 \leq i_1, \dots , i_m \leq n.$
\end{center}
(A square matrix of order $n$ is a $2$-order $n$-dimensional square tensor.)
An $m$-order $n$-dimensional square real tensor is said to be a nonnegative (positive) tensor if all its entries are nonnegative (positive).
 In  \cite{lim} and \cite{qi1},   the notion of eigenvalues of some classes of tensors was introduced independently. In \cite{cheng} and \cite{chang2}, a unified notion of eigenvalues was given for all $m$-order $n$-dimensional tensors.  Subsequently a lot of work has been done in the spectral theory of tensors. In particular, a Perron-Frobenius theory was developed for nonnegative tensors in \cite{cheng}, \cite{Shmuel}, \cite{ya-ya2} and \cite{ya-ya1}, see also the survey article \cite{cheng1}. In \cite{lim}, the notion of irreducible tensors was defined and in \cite{Shmuel} and \cite{hu}, the notion of weak irreducibility. Monotonicity of the spectral radius of nonnegative tensors and strict monotonicity of the spectral radius of irreducible nonnegative tensors were proved in \cite{ya-ya2}, \cite{ya-ya1}. Here, we prove strict monotonicity of the spectral radius of weakly irreducible nonnegative tensors. This is done in Section \ref{weak}.

The concepts of   positive (semi)definite matrices, $P(P_0)$-matrices and $M$-matrices were extended to tensors  in  \cite{qi1}, \cite{song_qi} and \cite{qi-zh-zho}, respectively. When modeling a real life problem the data may contain some errors because of inaccuracy of measurements, noises, round-off errors, etc. In this context it is natural to consider matrices or tensors whose  entries are from some intervals. If $\mathcal{B}-\mathcal{A}$ is a nonnegative tensor ($\mathcal{A} \leq \mathcal{B}$), then the interval hull  and the interior of the interval hull of the tensors $\mathcal{A}$ and $\mathcal{B}$ are defined as follows:
\begin{itemize}
 \item $I(\mathcal{A}, \mathcal{B}) = \{\mathcal{C} : C_{i_1\dots i_m} = t_{i_1\dots i_m} A_{i_1\dots i_m} + (1 - t_{i_1\dots i_m}) B_{i_1\dots i_m}, ~ t_{i_1\dots i_m} \in [0, 1]  ~\mbox{for all}~ i_1, \dots , i_m \in \{1,\dots ,n\}\}$,
  \item $int(I(\mathcal{A}, \mathcal{B})) = \{\mathcal{C} : C_{i_1\dots i_m} = t_{i_1\dots i_m} A_{i_1\dots i_m} + (1 - t_{i_1\dots i_m}) B_{i_1\dots i_m}, ~ t_{i_1\dots i_m} \in (0, 1) ~\mbox{for all}~ i_1, \dots , i_m \in \{1,\dots ,n\}\}.
  $

\end{itemize}
Note that $int (I(\mathcal{A}, \mathcal{B}))$ need not be the  topological interior of $I(\mathcal{A}, \mathcal{B})$ in $\mathbb{R}^{n^m}$.
 In \cite{rajkcs} interval hull of $M$-matrices was studied. In Section \ref{interval_m_tensors} we extend these results for $\mathcal{M}$-tensors.
In \cite{rohn1} and \cite{rohn2}  interval hull of positive (semi)definite matrices and $P(P_0)$-matrices was studied. For each of these classes a necessary and sufficient condition for an interval hull of matrices to be in the class was given.   In Section \ref{positive}, we extend these results to interval hull of positive (semi)definite tensors and $P(P_0)$-tensors.

\section{Notation, definitions and known results}\label{definitions}

Let $\mathbb{R}^{n}(\mathbb{C}^n)$ denote the $n$-dimensional real (complex) vector space. Let $\mathbb{R}^{n}_+ ~ (\mathbb{R}^n_{++})$ denote the set of all real $n$-tuples with nonnegative (positive) entries respectively. Vectors  are denoted by lower case letters ($x, y, \dots$), matrices by upper case letters ($A, B, \dots$) and tensors by calligraphic capital letters ($\mathcal{A}, \mathcal{B}, \dots$). The $i^{th}$ entry of a vector $x$ is denoted by $x_i$, the $(i,j)^{th}$ entry of a matrix $A$ is denoted by $A_{ij}$  and the $(i_1,\dots, i_m)^{th}$ entry of a tensor $\mathcal{A}$ is denoted by $A_{i_1\dots i_m}$.
For two $m$-order $n$-dimensional real tensors $\mathcal{A}$ and $ \mathcal{B}$, we write  $\mathcal{A} \leq \mathcal{B}$ if $A_{i_1\dots i_m} \leq B_{i_1\dots i_m}$ for all $1 \leq i_1, \dots , i_m \leq n$ and $\mathcal{A} < \mathcal{B}$ if $A_{i_1\dots i_m} < B_{i_1\dots i_m}$ for all $1 \leq i_1, \dots , i_m \leq n.$
For a subset $\alpha$ of $\{1, \dots ,n\}$,  $|\alpha|$ denotes the number of elements of $\alpha$.
\begin{defn}
{\rm A tensor $\mathcal{A}$ is said to be \emph{symmetric} if its entries are invariant under any permutation of the indices $\{i_1, \dots, i_m\}$.}
\end{defn}
A particular example is the $m$-order $n$-dimensional identity tensor, denoted by $\mathcal{I}=(I_{i_1\dots i_m})$, defined as follows:$$I_{i_1\dots i_m} =
\left\{
	    \begin{array}{ll}
		1 & \mbox{if } i_1 = \dots = i_m, \\
		0 & \mbox{otherwise}.
	\end{array}
\right.$$

\begin{defn}
{\rm Let $\mathcal{A}=(A_{i_1\dots i_m})$ be an $m$-order $n$-dimensional tensor and $\alpha \subset\{1, \dots , n\}$ with $|\alpha| = r$.
A \emph{principal subtensor} $\mathcal{A}[\alpha]$ of the tensor $\mathcal{A}$ with index set $\alpha$ is an $m$-order $r$-dimensional subtensor of  $\mathcal{A}$ consisting of $r^m$ elements defined as
follows:

\begin{center}
$\mathcal{A}[\alpha] = (A_{i_1\dots i_m})$ , where $i_1, \dots , i_m \in \alpha$.
\end{center}
}
\end{defn}

In \cite{lim}, the notion of irreducible tensors was introduced.

\begin{defn}
{\rm
An $m$-order $n$-dimensional tensor $\mathcal{A} = (A_{i_1\dots i_m})$  is called \emph{reducible} if there exists a nonempty proper subset $\alpha \subset \{1, \dots ,n \}$ such that
\begin{center}
$A_{i_1\dots i_m} = 0$ for all $i_1 \in \alpha $ and $i_2, \dots ,i_m \notin \alpha$.
\end{center}
A tensor $\mathcal{A}$ is said to be \emph{irreducible} if it is not reducible.
}\end{defn}

In \cite{Shmuel} and \cite{hu},  the notion of weakly irreducible nonnegative tensors  was introduced. With a nonnegative tensor $\mathcal{A}= (A_{i_1\dots i_m})$, we associate the nonnegative $n \times n $ matrix $R\mathcal{(A)}$:
\begin{center}
$R\mathcal{(A)}_{ij} = \sum_{ \{i_2, \dots, i_m\} \ni j} A_{ii_2\dots i_m}$.
\end{center}

\begin{defn}
{\rm
A nonnegative tensor $\mathcal{A}= (A_{i_1\dots i_m})$ is said to be \emph{weakly reducible} if $R\mathcal{(A)}$ is a reducible matrix. It is  \emph{weakly irreducible} if it is not weakly reducible.
}\end{defn}

The following result holds:
\begin{prop}  \cite[Lemma 2.1]{cheng},  \cite[Lemma 3.1]{Shmuel} \label{irr_imply_weakirr}
Let $\mathcal{A}$ be a nonnegative tensor. If $\mathcal{A}$ is irreducible, then $\mathcal{A}$ is weakly irreducible.
\end{prop}

The converse is true only for matrices, and not for higher order tensors. For $ x \in \mathbb{C}^n $ and a natural number $k$, the vector $x^{[k]}$  is the Hadamard power of $x$, i.e. $x^{[k]}_i = x^{k}_i$ for all $i$. For an $m$-order $n$-dimensional tensor $\mathcal{A}$ and for a vector $x \in \mathbb{C}^n$,  $\mathcal{A}x^{m-1}$ is the  vector in $\mathbb{C}^n$ defined by $(\mathcal{A}x^{m-1})_i = \sum_{i_2, \dots , i_m = 1}^{n}A_{ii_2\dots i_m} x_{i_2} \cdots x_{i_m}$.
Now we recall the definition of eigenvalue of a tensor.

\begin{defn}
{\rm
If a pair $(\lambda, x) \in \mathbb{C} \times \mathbb{C}^n\setminus\{0\}$ satisfies the equation
$\mathcal{A}x^{m-1}= \lambda x^{[m-1]}$, then $\lambda$ is called an \emph{eigenvalue} of $\mathcal{A}$ and $x$ is called an \emph{eigenvector} corresponding to the eigenvalue $\lambda$. If $(\lambda, x) \in \mathbb{R} \times \mathbb{R}^n\setminus \{0\}$, then $\lambda$ is called an \emph{$H$-eigenvalue} of $\mathcal{A}$. If $(\lambda, x) \in \mathbb{R} \times \mathbb{R}^n_{+} \setminus \{0\}$, then $\lambda$ is called an \emph{$H^+$-eigenvalue} of $\mathcal{A}$. If $(\lambda, x) \in \mathbb{R} \times \mathbb{R}^n_{++}$, then $\lambda$ is called an \emph{$H^{++}$-eigenvalue} of $\mathcal{A}$. The \emph{spectral radius} $\rho(\mathcal{A})$ of a tensor $\mathcal{A}$ is defined to be $\max \{|\lambda| : \lambda~ \mbox{is an eigenvalue of}~ \mathcal{A}\}$.
}
\end{defn}

The following  result is an analogue for tensors of  a special case of the spectral mapping theorem of matrices.

\begin{prop} \cite[Corollary 3]{qi1}\label{spec_poly}
Let $\mathcal{A}$ be an $m$-order $n$-dimensional tensor. Suppose that $\mathcal{B}= a(\mathcal{A} + b \mathcal{I})$, where $a$ and $b$ are two real numbers. Then $\mu$ is an eigenvalue of $\mathcal{B}$ if and only if $\mu = a(\lambda + b)$, where $\lambda$ is an eigenvalue of $\mathcal{A}$. In this case, they have the same eigenvectors.
\end{prop}

The next important result is a Perron-Frobenius Theorem for nonnegative tensors. It combines  \cite[Theorem 1.3]{cheng} and  \cite[Theorem 4.1]{Shmuel}.
\begin{prop}\label{perron_frob}
Let $\mathcal{A}$ be an $m$-order $n$-dimensional nonnegative tensor.
\begin{itemize}
\item[(a)] Then the spectral radius $\rho(\mathcal{A})$ is an $H^{+}$-eigenvalue of $\mathcal{A}$.
\item[(b)] If $\mathcal{A}$ is weakly irreducible, then $\rho(A)$ is an $H^{++}$-eigenvalue of $\mathcal{A}$ and no other eigenvalue has a positive eigenvector.
\end{itemize}
\end{prop}

The following  is a Collatz-Wielandt type result for weakly irreducible nonnegative tensors.

\begin{prop}\cite[Corollary 4.2]{Shmuel}, \cite[Theorem 4.1]{hu}\label{fried_colla}
Let $\mathcal{A}=(A_{i_1\dots i_m})$ be an $m$-order $n$-dimensional  weakly irreducible nonnegative tensor. Then the unique scalar $\lambda$, for which there is a vector $z \in \mathbb{R}^n_{++}$ with  $\mathcal{A}z^{m-1} = \lambda z^{[m-1]}$, satisfies:
\begin{center}
$\lambda = \sup_{x \in \mathbb{R}^n_{+}\setminus \{0\}} \min_{x_i>0} \frac{(\mathcal{A} x^{m-1})_i}{x^{m-1}_i} = \inf_{x \in \mathbb{R}^n_{++}} \max_{i} \frac{(\mathcal{A} x^{m-1})_i}{x^{m-1}_i} $.
\end{center}
\end{prop}

The following result deals with the monotonicity of the spectral radius of nonnegative tensors.

\begin{prop}\cite[Lemma 3.5]{ya-ya1}  \cite[Theorem 2.20]{ya-ya2} \label{mono}
Let $\mathcal{A}$ and $\mathcal{B}$ be two $m$-order $n$-dimensional tensors such that $0 \leq \mathcal{A} \leq \mathcal{B}$. Then
\begin{itemize}
\item[(a)] $\rho(\mathcal{A}) \leq \rho(\mathcal{B})$.
\item[(b)] if $\mathcal{B}$ is irreducible and $\mathcal{A} \neq \mathcal{B}$, then  $\rho(\mathcal{A}) < \rho(\mathcal{B})$.
\end{itemize}
\end{prop}

The following result establishes  a relation between the spectral radius of a tensor and its principal subtensors.

\begin{prop} \cite[Lemma 2.2]{hu}\label{spec_sub}
Let  $\mathcal{A}[\alpha]$ be an $m$-order $r$-dimensional principal subtensor of an $m$-order $n $-dimensional nonnegative tensor $\mathcal{A}$. Then $\rho(\mathcal{A}[\alpha]) \leq \rho(\mathcal{A})$.
\end{prop}

The following result is the analogue of the Frobenius normal form of nonnegative matrices to weak irreducible nonnegative tensors.

\begin{prop} \cite[Theorem 5.7]{hu} \label{weak_irr_block}
Let $\mathcal{A}$ be an $m$-order $n$-dimensional nonnegative tensor. If $\mathcal{A}$ is weakly reducible, then there exists a partition $\{\alpha_1, \dots, \alpha_k\}$ of $\{1 , \dots, n\}$ such that every tensor in $\{\mathcal{A}[\alpha_j] : j \in \{1, \dots , k\}\}$ is weakly irreducible and $A_{ri_2 \dots i_m} = 0$ for all $r \in \alpha_p$ , $i_j \in \alpha_q$ for some $j \in \{2, \dots, m\}$ and $p > q$.
\end{prop}

The next result establishes the relation between  the  spectral radius of a nonnegative tensor and that of its weakly irreducible subtensors considered as in the above theorem.
\begin{prop}\cite[Theorem 5.8]{hu}\label{weak_irr_block_spec}
Let $\mathcal{A}$ be an $m$-order $n$-dimensional weakly reducible nonnegative, and let $\{\alpha_1, \dots , \alpha_k\}$ be a partition of $\{1, \dots, n\}$ determined by Proposition \ref{weak_irr_block}. Then $\rho(\mathcal{A}) = \rho(\mathcal{A}[\alpha_p])$ for some $p \in \{1, \dots, k\}$.
\end{prop}

There is an analogue of the Frobenius normal form also for reducible nonnegative tensors. However, there is no result like Proposition \ref{weak_irr_block_spec} in that case. We refer to  \cite[Example $5.5$]{hu}.

\section{Weakly irreducible nonnegative tensors}\label{weak}

In this section we first recall (Proposition \ref{ya-ya-irre}) a result about the dominant eigenvalue of an irreducible nonnegative tensor. We then prove a similar result for weakly irreducible nonnegative tensors, and use it in Theorem \ref{irre_mono} to establish a strict monotonicity of the spectral radius of weakly irreducible nonnegative tensors.

\begin{prop} \cite[Lemma 3.1]{ya-ya1}\label{ya-ya-irre}
Let $\mathcal{A}$ be an $m$-order $n$-dimensional irreducible nonnegative tensor. If there exists $y \in \mathbb{R}^n_{+}, y \neq 0$, such that $(\mathcal{A -\rho(A)\mathcal{I}})y^{m-1} \geq 0$, then $y$ is an eigenvector of $\mathcal{A}$ corresponding to $\rho{(\mathcal{A})}$.

\end{prop}

We can now state and prove the analogue of the above proposition for weakly irreducible nonnegative tensors.

\begin{lemma}\label{iterate}
Let $\mathcal{A}$ be an $m$-order $n$-dimensional weakly irreducible nonnegative tensor. If there exists $y \in \mathbb{R}^n_{++}$, such that $(\mathcal{A -\rho(A)\mathcal{I}})y^{m-1} \geq 0$, then $y$ is an eigenvector of $\mathcal{A}$ corresponding to the eigenvalue $\rho{(\mathcal{A})}$.
\end{lemma}
\begin{proof}
We have to prove that $(\mathcal{A}- \rho(\mathcal{A})\mathcal{I})y^{m-1} = 0$. First we shall prove that $(\mathcal{A}- \rho(\mathcal{A})\mathcal{I})y^{m-1} > 0$ is not possible. Suppose $(\mathcal{A}- \rho(\mathcal{A})\mathcal{I})y^{m-1} > 0$, then there exists $\epsilon > 0$ such that $\mathcal{A}y^{m-1} \geq (\rho(A) + \epsilon)\mathcal{I}y^{m-1}$. Then by Propositions \ref{perron_frob} and  \ref{fried_colla}, we have
\begin{center}
$\rho({\mathcal{A}}) = \sup_{z \in \mathbb{R}^n_+ \setminus \{0\}} \min_{z_i >0} \frac{(\mathcal{A}z^{m-1})_i}{z_i^{m-1}} \geq \min_{i} \frac{(\mathcal{A}y^{m-1})_i}{y_i^{m-1}} \geq \rho(\mathcal{A}) + \epsilon$,
\end{center}
a contradiction. So $(\mathcal{A}- \rho(\mathcal{A})\mathcal{I})y^{m-1}$ has at least one zero entry.

Now, for $i = 1, \dots , n $, define the functions $f_i : \mathbb{R}^n \longrightarrow \mathbb{R}$ as follows:
\begin{center}
$f_{i}(x) = \sum_{i_2,\dots ,i_m = 1}^{n} A_{ii_2\dots i_m} x_{i_2}\cdots x_{i_{m}} - \rho(\mathcal{A}) x_i^{m-1}$.
$ $
\end{center}

If $f_i(y) = 0$ for all $1 \leq i \leq n$, then we are done. Suppose $f_i(y)\neq 0$ for some $i$.
Without loss of generality assume $f_{i}(y) > 0$ for $1 \leq i \leq k$ and $f_i (y) = 0$ for $k+1 \leq i \leq n$.
Since $f_i$ is continuous and  $f_i(y) > 0$ for all $1 \leq i \leq k$, there exists $\delta > 0$ such that $f_{i}(x) > 0$ for all $x \in B(y, \delta) = \{(z_1,\dots, z_n) \in \mathbb{R}^n : \sum_{i=1}^{n}|y_i - z_i| < \delta\}$ and $1 \leq i \leq k$.


Let $z \geq y > 0$ such that $z_{j} = y_{j}$ for all $j \in \{k+1, \dots , n \}$ and $z_j > y_j$ for all $j \in \{1, \dots , k\}$. Then, we have $z_{i_2}\cdots z_{i_m} - y_{i_2}\cdots y_{i_m} > 0$ for all $i_2, \dots, i_m$ with at least one $i_j \notin \{ k+1, \dots , n \}$. We claim that $f_{i}(z) > f_{i} (y)$ for some $i \in \{k+1, \cdots , n\}$. If $f_{i}(z) = f_{i} (y)$ for all $i \in \{k+1, \dots , n\}$, then $A_{ii_2\dots i_m} = 0$ for all $i \in \{k+1, \dots , n\}$ and $i_2, \dots, i_m $ with at least one $i_j \notin \{k+1 , \dots , n\} $.
Thus $R(\mathcal{A})_{ij} = 0$ for all $i \in \{k+1 , \dots , n\}$ and $j \notin \{k+1 , \dots , n\}$. Hence $\mathcal{A}$ is weakly reducible, a contradiction. Thus $f_{i}(z) > 0$ for some $i \in \{k+1, \dots , n\}$.

Choose $z$ such that $z \in B(y, \delta)$ and $z \geq y $ with $z_{j} = y_{j}$ for all $j \in \{k+1, \dots , n \}$ and $z_j > y_j$ for all $j \in \{1, \dots , k\}$. Without loss of generality assume $f_{k+1}(z) > 0$. If we replace $y$ by $z$, then $f_i(z) > 0$ for all $i \in \{1, \dots ,k+1 \}$, $f_i(z)\geq 0$ for all $i \in \{k+2, \dots , n\}$ and $z \in \mathbb{R}^n_{++}$. Thus repeating the process at most $(n-k)$ times, we get that for some $z^{'} \in \mathbb{R}^n_{++}$, $f_i(z^{'}) > 0$  for all $1 \leq i \leq n$, which is, as shown above, not possible.

Hence $(\mathcal{A}- \rho(\mathcal{A})I)y^{m-1} = 0$, that is, $y$ is an eigenvector corresponding to $\rho(\mathcal{A})$.
\end{proof}

The proof of next lemma is similar to the that of the above lemma by using the fact $\rho(A)=\inf_{x\in \mathbb{R}^n_{++}}\max_{1\le i\le n}\frac{(\mathcal{A}x^{m-1})_i}{x_i^{m-1}}$ from proposition \ref{fried_colla}.

\begin{lemma}\label{iterate1}
Let $\mathcal{A}$ be an $m$-order $n$-dimensional weakly irreducible tensor. If $y\in \mathbb{R}^n_{++}$
satisfies $(\mathcal{A} -\rho(\mathcal{A})\mathcal{I})y^{m-1}\le 0$, then $y$ is a
positive eigenvector corresponding to $\rho(\mathcal{A})$.
\end{lemma}

In the next theorem we establish a relation between the spectral radius of a weakly irreducible nonnegative tensor and its principal subtensors.
\begin{theorem}\label{prin-str-mono}
Let $\mathcal{A}$ be an $m$-order $n$-dimensional weakly irreducible tensor, and $\mathcal{A}[\alpha]$ a  principal subtensor of $\mathcal{A}$,
$|\alpha|=k<n$. Then $\rho(\mathcal{A}[\alpha])<\rho(\mathcal{A})$.
\end{theorem}

\begin{proof}
Since $\rho(\mathcal{A}[\alpha])$ is equal to $\rho(\mathcal{A}')$ for some weakly irreducible subtensor $\mathcal{A}'$ of $\mathcal{A}[\alpha]$, it suffices to consider the case  that $\mathcal{A}[\alpha]$ itself is weakly irreducible.
Let $z\in \mathbb{R}^k_{++}$ and $y\in \mathbb{R}^n_{++}$ be  positive
eigenvectors of $\mathcal{A}[\alpha]$ and $\mathcal{A}$, respectively,  corresponding to their  largest eigenvalues $\rho(\mathcal{A}[\alpha])$ and $\rho(\mathcal{A})$.
Let $\bar{y}= y[\alpha]$. Then
\[\rho(\mathcal{A}[\alpha])=\inf_{x\in \mathbb{R}^k_{++}}\max_{1\le i\le k}\frac{(\mathcal{A}[\alpha]x^{m-1})_i}{x_i^{m-1}}\le \max_{1\le i\le k}\frac{(\mathcal{A}[\alpha]\bar{y}^{m-1})_i}{\bar{y}_i^{m-1}}\le \max_{1\le i\le n}\frac{(\mathcal{A}y^{m-1})_i}{y_i^{m-1}}=\rho(\mathcal{A}). \]
The last inequality uses the fact that $(\mathcal{A}[\alpha] \bar{y}^{m-1})_i\le (\mathcal{A}y^{m-1})_i$ for every $1\le i\le k$.
 If $\rho(\mathcal{A}[\alpha])=\rho(\mathcal{A})$, then all the above inequalities  are equalities. In particular,
\[\rho(\mathcal{A}[\alpha])=\rho(\mathcal{A})=\max_{1\le i\le k}\frac{(\mathcal{A}[\alpha]\bar{y}^{m-1})_i}{\bar{y}_i^{m-1}}.\]
By the previous lemma, $\bar{y}$ is the unique (up to scalar multiple) eigenvector of the weakly irreducible tensor $\mathcal{A}[\alpha]$.
We may therefore assume that $\bar{y}_i=z_i$ for every $1\le i\le k$. Since $\mathcal{A}$ is weakly irreducible, there exist $1\le i\le k$ and $i_2, \ldots, i_m\in \{1, \dots, n\}$ at least
one of which is greater than $k$, such that $A_{ii_2\dots i_m}>0$. Thus using this specific $i$ we get that
\[\rho(\mathcal{A}[\alpha])=\frac{(\mathcal{A}[\alpha] \bar{y}^{m-1})_i}{\bar{y}_i^{m-1}}<\frac{(\mathcal{A}y^{m-1})_i}{y_i^{m-1}}=\rho(\mathcal{A}).\]
\end{proof}

Now we are ready  to prove the strict monotonicity of the spectral radius for weakly irreducible nonnegative tensors.

\begin{theorem}\label{irre_mono}
Let $\mathcal{A}$ and $\mathcal{B}$ be two tensors such that $ 0 \leq \mathcal{A} \leq \mathcal{B}$ and $\mathcal{B}$ is a weakly irreducible tensor.  Then,
\begin{itemize}
\item[(a)] $\rho(\mathcal{A}) \leq \rho(\mathcal{B})$,
\item[(b)] if $\mathcal{A} \neq \mathcal{B}$, then $\rho(\mathcal{A}) < \rho(\mathcal{B})$.

\end{itemize}

\end{theorem}

\begin{proof}
\textbf{Case 1:}
Let $\mathcal{A}$ be a nonnegative weakly irreducible tensor, then by Proposition \ref{perron_frob} we have that $\rho(\mathcal{A})$ is an $H^{++}$-eigenvalue of $\mathcal{A}$ with an eigenvector $y \in \mathbb{R}^n_{++}$. Since $\mathcal{A} \leq \mathcal{B}$, we have $\rho(\mathcal{A})y^{[m-1]} = \mathcal{A}y^{m-1} \leq \mathcal{B}y^{m-1}$.

Now by Proposition \ref{fried_colla}, we have $$\rho(\mathcal{B}) = \sup_{x \in \mathbb{R}^n_{+} \setminus \{0\}} \min_{x_i >0} \frac{(\mathcal{B}x^{m-1})_i}{x_i^{m-1}} \geq \min_{i} \frac{(\mathcal{B}y^{m-1})_i}{y_i^{m-1}}  \geq \min_{i} \frac{(\mathcal{A}y^{m-1})_i}{y_i^{m-1}} =  \rho(\mathcal{A})$$

Suppose $\rho({\mathcal{B}}) = \rho(\mathcal{A})$. Then we have $\mathcal{B}y^{m-1} \geq \rho(\mathcal{B})y^{[m-1]}$. Now by Lemma \ref{iterate}, we have $\mathcal{B}y^{m-1} = \rho(\mathcal{B})y^{[m-1]}$. Thus $\mathcal{B}y^{m-1} = \mathcal{A}y^{m-1}$. Since $y \in \mathbb{R}^n_{++}$, we have $\mathcal{B}= \mathcal{A}$. Thus, if $\mathcal{A} \neq \mathcal{B} $, then $\rho({\mathcal{A}}) < \rho(\mathcal{B})$.

\textbf{Case 2:} Let $\mathcal{A}$ be weakly reducible. Then there exists a weakly irreducible subtensor  of $\mathcal{A}$, $\mathcal{A}[\alpha]$, such that $\rho(\mathcal{A}[\alpha])=\rho(\mathcal{A})$ and $|\alpha|=k<n$. Then
$\mathcal{B}[\alpha]$ is a weakly irreducible subtensor of $\mathcal{B}$. By
 Theorem \ref{prin-str-mono},
\[\rho(\mathcal{A})=\rho(\mathcal{A}[\alpha])\le \rho(\mathcal{B}[\alpha])<\rho(\mathcal{B}).\]
\end{proof}

Theorem \ref{irre_mono} implies the previously known result for irreducible tensors (see proposition 2.5).
\begin{rem}
From Theorem \ref{prin-str-mono}, it is clear that any nonnegative eigenvector of a weakly irreducible nonnegative tensor corresponding to the spectral radius must be positive.
\end{rem}
\section{Interval hull of $\mathcal{M}$-tensors}\label{interval_m_tensors}

In this section,  we first recall the definition of $\mathcal{M}$-tensors. In Theorem \ref{principal}, we prove that a principal subtensor of an $\mathcal{M}$-tensor (a strong $\mathcal{M}$-tensor) is an $\mathcal{M}$-tensor (a strong $\mathcal{M}$-tensor).  Then we prove some properties of $\mathcal{M}$-tensors, and use them to prove one of the main results of this section, Theorem \ref{main_int} about the interval hull of strong $\mathcal{M}$-tensors.

We recall the definitions of $\mathcal{Z}$-tensors, $\mathcal{M}$-tensors and strong $\mathcal{M}$-tensors from \cite{qi-ding-wei} and \cite{qi-zh-zho}.

\begin{defn}
{\rm
Let $\mathcal{A}$ be an $m$-order $n$-dimensional tensor. Then $\mathcal{A}$ is called a \emph{$\mathcal{Z}$-tensor} if there exists a nonnegative tensor $\mathcal{D}$ and a real number $s$ such that $\mathcal{A} = s\mathcal{I} - \mathcal{D}$. A $\mathcal{Z}$-tensor $\mathcal{A} = s \mathcal{I} - \mathcal{D}$ is said to be an \emph{$\mathcal{M}$-tensor} if $s \geq \rho(\mathcal{D})$. If $s > \rho(\mathcal{D})$, then $\mathcal{A}$ is called a \emph{strong $\mathcal{M}$-tensor}. A $\mathcal{Z}$-tensor $\mathcal{A} = s\mathcal{I} - \mathcal{D}$ with $\mathcal{D}\geq 0$, is called \emph{weakly irreducible}, if $\mathcal{D}$ is weakly irreducible.
}
\end{defn}

It is well known that a principal submatrix of an $M$-matrix (invertible $M$-matrix) is an $M$-matrix (invertible $M$-matrix). In the following theorem we prove the same holds for $\mathcal{M}$-tensors.
\begin{theorem}\label{principal}
Let $\mathcal{A}$ be an $m$-order $n$-dimensional $\mathcal{M}$-tensor  (a strong $\mathcal{M}$-tensor). Then all principal subtensors of $\mathcal{A}$ are $\mathcal{M}$-tensors (strong $\mathcal{M}$-tensors).
\end{theorem}
\begin{proof}
Let $\mathcal{A}[\alpha]$ be an $m$-order $r$-dimensional principal subtensor $\mathcal{A}$, where $\alpha \subseteq \{1,\dots,n\}$. Let $\mathcal{A} = s \mathcal{I} - \mathcal{D}$ with $\mathcal{D} \geq 0 $ and $s \geq \rho(\mathcal{D})$. Consider the following decomposition  $\mathcal{A}[\alpha] = s \mathcal{I} - \mathcal{D}[\alpha]$. Then $\mathcal{D}[\alpha]$ is a principal subtensor of $\mathcal{D}$. Now, by Lemma \ref{spec_sub} we have $\rho(\mathcal{D}[\alpha]) \leq \rho(\mathcal{D})$. But $s \geq \rho(\mathcal{D})$ and hence $s \geq  \rho (\mathcal{D}[\alpha])$. Thus $\mathcal{A}[\alpha]$ is an $\mathcal{M}$-tensor. The proof for the case of  strong $\mathcal{M}$-tensors is similar.
\end{proof}

Since each diagonal entry of a tensor is a principal subtensor, Theorem \ref{principal} yields a new proof to the following known result.

\begin{cor}\cite[Propositions 4 and 15]{qi-ding-wei}
The diagonal entries of an $\mathcal{M}$-tensor (a strong $\mathcal{M}$-tensor) are nonnegative (positive).
\end{cor}

The following observation will be used in the study of interval hull of $\mathcal{M}$-tensors.

\begin{obs}\label{rep_inde}
Let $\mathcal{A}$ be an $\mathcal{M}$-tensor. If $\mathcal{A} = t \mathcal{I} - \mathcal{E}$ for some $t \geq 0$ and $\mathcal{E} \geq 0$, then $\rho(\mathcal{E}) \leq t$. If $\mathcal{A}$ is a strong $\mathcal{M}$-tensor and  $\mathcal{A} = t\mathcal{I} - \mathcal{E}$ for some $t \geq 0$  and $\mathcal{E} \geq 0$, then $\rho(\mathcal{E}) < t$.
\end{obs}

\begin{proof}
 Let $\mathcal{A} = s \mathcal{I} - \mathcal{D}$ with $\mathcal{D} \geq 0$ and $s \geq \rho(\mathcal{D})$. Let $\mathcal{A} = t \mathcal{I} - \mathcal{E}$.
  If $s \leq t$, then $\mathcal{A} = t\mathcal{I} - \mathcal{E} = t\mathcal{I} -( (t-s)\mathcal{I} + \mathcal{D})$ and $(t-s)\mathcal{I} + \mathcal{D} \geq 0$. Now, it follows from Proposition \ref{spec_poly} and the nonnegativity of $\mathcal{D}$ that $(t-s) + \rho(\mathcal{D}) = \rho(\mathcal{E})$ and hence $\rho(\mathcal{E}) \leq t$. If $s \geq t$, then $\mathcal{A} = s \mathcal{I} - ((s-t)\mathcal{I}+\mathcal{E})$ and $(s-t)\mathcal{I}+\mathcal{E} \geq 0$. Again by Proposition \ref{spec_poly}, we have $(s-t)+\rho(\mathcal{E}) = \rho(\mathcal{D})$ and hence $\rho(\mathcal{E}) \leq t$. The proof for a strong $\mathcal{M}$-tensor is similar.
\end{proof}

In the next theorem we prove that if a $\mathcal{Z}$-tensor $\mathcal{B}$ is greater than or equal to an $\mathcal{M}$-tensor entry-wise, then $\mathcal{B}$ is also an $\mathcal{M}$-tensor.

\begin{theorem}\label{strong_int}
Let $\mathcal{A}$ and $\mathcal{B}$ be  $m$-order $n$-dimensional $\mathcal{Z}$-tensors such that $\mathcal{A} \leq \mathcal{B}$.
 \begin{itemize}
 \item[(a)]Suppose $\mathcal{A}$ is an $\mathcal{M}$-tensor (a strong $\mathcal{M}$-tensor). Then $\mathcal{B}$ is also an  $\mathcal{M}$-tensor (a strong $\mathcal{M}$-tensor).
\item[(b)] Suppose $\mathcal{A}$ is a  weakly irreducible $\mathcal{M}$-tensor and $\mathcal{A} \neq \mathcal{B}$. Then $\mathcal{B}$ is a strong $\mathcal{M}$-tensor.
\end{itemize}
\end{theorem}

\begin{proof}
\begin{itemize}
\item[(a)] Let $A$ be an $\mathcal{M}$-tensor. Then $\mathcal{A} = s \mathcal{I} - \mathcal{D}$, where $\mathcal{D} \geq 0$ and $s \geq \rho(\mathcal{D})$.
 Choose $t \geq s$  and $ \mathcal{E} \geq 0$ such that  $\mathcal{B} = t \mathcal{I} - \mathcal{E}$. Then  we have $\mathcal{A} = t \mathcal{I} - \mathcal{D}^{'}$ where $\mathcal{D}^{'} \geq \mathcal{E} \geq 0$.   Thus $\rho(\mathcal{D}^{'}) \geq \rho(\mathcal{E})$. By Observation \ref{rep_inde}, we have $t \geq \rho(\mathcal{D}^{'})$, so $t \geq \rho(\mathcal{E})$. Hence $\mathcal{B}$ is an $\mathcal{M}$-tensor. The proof for strong $\mathcal{M}$-tensors is similar.

\item[(b)] We have  $\mathcal{D}^{'}$ is weakly irreducible. Hence by Observation \ref{rep_inde} and Theorem \ref{irre_mono}, we have $t \geq \rho(\mathcal{D}^{'})> \rho(\mathcal{E})$. Thus $\mathcal{B}$ is a strong $\mathcal{M}$-tensor.
\end{itemize}\end{proof}

\begin{cor} \cite[Theorem 3.13]{qi-zh-zho}
Let $\mathcal{D}$ be any nonnegative diagonal tensor. If $\mathcal{A}$ is a symmetric  $\mathcal{M}$-tensor (strong $\mathcal{M}$-tensor), then $\mathcal{A}+\mathcal{D}$ is also a symmetric  $\mathcal{M}$-tensor (strong $\mathcal{M}$-tensor).
\end{cor}

In the following theorem we give a sufficient condition for the interior of the interval hull of two tensors to be a subset of the class of strong $\mathcal{M}$-tensors. It is an application of the Frobenius normal form for  weakly irreducible nonnegative tensors  and the strict monotonicity of the spectral radius of weakly irreducible nonnegative tensors.
\begin{theorem}\label{main_int}
Let $\mathcal{A}$ and $\mathcal{B}$ be two $m$-order $n$-dimensional tensors. Suppose that $\mathcal{A} \leq \mathcal{B}$, $\mathcal{A}$ is an $\mathcal{M}$-tensor and $\mathcal{B}$ is a strong $\mathcal{M}$-tensor. Then $\mathcal{C}$ is a  strong $\mathcal{M}$-tensor for all $\mathcal{C} \in int(I(\mathcal{A}, \mathcal{B}))$
\end{theorem}

\begin{proof}

Let $\mathcal{A}$  be an $\mathcal{M}$-tensor and $\mathcal{B}$ be a strong $\mathcal{M}$-tensor. If $\mathcal{A}$ is a strong $\mathcal{M}$-tensor, then by Theorem \ref{strong_int}, the elements of $I(\mathcal{A},\mathcal{B})$ are strong $\mathcal{M}$-tensors.

Without loss of generality assume $\mathcal{A}= t \mathcal{I} - \mathcal{D}$ and $\mathcal{B}= t \mathcal{I} - \mathcal{E}$,  $t > \rho(\mathcal{E})$, $t = \rho{(\mathcal{D})}$ and $\mathcal{E},\mathcal{D} \geq 0$. Let $\mathcal{C} \in int(I(\mathcal{A},\mathcal{B}))$, then $\mathcal{C} = t \mathcal{I} - \mathcal{F}$ such that $\mathcal{F} \geq 0$. Since $\mathcal{A} \leq \mathcal{C}$ and $\mathcal{A} \neq \mathcal{C}$, we have  $\mathcal{D} \geq \mathcal{F}$ and $\mathcal{D} \neq \mathcal{F}$.

Suppose $\mathcal{F}$ is weakly irreducible. Then by Theorem \ref{strong_int}, $\mathcal{C}$ is a strong $\mathcal{M}$-tensor.

Suppose $\mathcal{F}$ is weakly reducible. Then by Propositions \ref{weak_irr_block} and \ref{weak_irr_block_spec}, there exists a partition $\{\alpha_1, \dots , \alpha_k\}$ of $\{1, \dots, n\}$ such that the principal subtensors of $\mathcal{F}$ corresponding to $\alpha_i$ are weakly irreducible for all $i = 1, \dots ,n$ and $\rho(\mathcal{F}) = \rho(\mathcal{F}[\alpha_j])$ for some $j$.  Since $\rho(\mathcal{D})= t$, by Proposition \ref{spec_sub} we have $t \geq \rho(\mathcal{D}[\alpha_j])$. If $\rho(\mathcal{F}[\alpha_j]) < \rho(\mathcal{D}[\alpha_j])$ or $t > \rho(\mathcal{D}[\alpha_j])$, then we are done. Otherwise,  $\rho(\mathcal{F}[\alpha_j]) = \rho(\mathcal{D}[\alpha_j])= t$. Then by Theorem \ref{irre_mono} we have $\mathcal{F}[\alpha_j] = \mathcal{D}[\alpha_j]$. As $\mathcal{C} \in int(I(\mathcal{A}, \mathcal{B}))$,  $\mathcal{F}[\alpha_j] = \mathcal{D}[\alpha_j]= \mathcal{E}[\alpha_j]$. But then $t =\rho(\mathcal{D}[\alpha_j])=\rho(\mathcal{F}[\alpha_j]) = \rho(\mathcal{E}[\alpha_j]) < t$, a contradiction.

 Thus we have that $\mathcal{C}$ is a strong $M$-tensor.
\end{proof}

\begin{rem}
In the above theorem, the condition that $\mathcal{B}$ is a strong $\mathcal{M}$-tensor is necessary. As is evident from the proof, if $\mathcal{B}$ is an $\mathcal{M}$-tensor but not a strong $\mathcal{M}$-tensor, then any $\mathcal{C} \in int (I(\mathcal{A}, \mathcal{B}))$ is an $\mathcal{M}$-tensor, but not a strong $\mathcal{M}$-tensor.
\end{rem}


\section{Interval hull of $P(P_0)$-tensors and  Positive (semi) definite tensors}\label{positive}
In this section, we establish  for each of the following classes positive (semi)definite tensors, $P(P_0)$-tensors, a necessary and sufficient condition for an interval hull of tensors to be a subset of the class.  We assert that for checking the positive (semi)definiteness of the entire interval hull it is enough to check the positive (semi)definiteness of only finitely many tensors in that interval. Similarly for $P(P_0)$-tensors. First we recall the definition of positive (semi)definite tensors.

 For an $m$-order $n$-dimensional tensor $\mathcal{A}$ and a vector $x \in \mathbb{C}$, the scalar $\mathcal{A}x^{m}$ is defined to be $\mathcal{A}x^m =\sum_{i_1, \dots , i_m =1}^{n} A_{i_1\dots i_m}x_{i_1}\cdots x_{i_m}$.

\begin{defn}[\cite{qi1}]
{\rm
Let $\mathcal{A}$ be an $m$-order $n$-dimensional tensor. Then $\mathcal{A}$ is said to be a \emph{positive semidefinite tensor}
if for any vector $x \in \mathbb{R}^n$, $\mathcal{A}x^m \geq 0$, and $\mathcal{A}$ is called a \emph{positive definite tensor} if for any nonzero vector $x \in \mathbb{R}^n$, $\mathcal{A}x^{m} >0$.
}
\end{defn}

From the definition it is clear that, if $m$ is odd, then there is no nontrivial positive semidefinite tensors. Next we recall the definition of $P(P_0)$-tensors.

\begin{defn}[\cite{song_qi}]
{\rm
Let $\mathcal{A}$ be an  $m$-order $n$-dimensional tensor. Then $\mathcal{A}$ is said to be a \emph{$P$-tensor}
if for any nonzero vector $x \in \mathbb{R}^n$, $x_i(\mathcal{A}x^{m-1})_i > 0$ for some $i \in \{1, \dots , n\}$, and $\mathcal{A}$ is called a \emph{$P_0$-tensor} if for any nonzero vector $x \in \mathbb{R}^n$, $x_i (\mathcal{A}x^{m-1})_i \geq 0$ for some $i \in \{1, \dots, n\}$ with $x_i \neq 0$.
}
\end{defn}

Let $\mathbf{Z}=\{(z_1,\dots,z_n)\in \mathbb{R}^n : z_i= \pm 1~\mbox{for all}~ i\in \{1,\dots,n\}\}$. For each $z =(z_1,\dots , z_n) \in \mathbf{Z}$
we define the matrix  $D_z$ to be  the diagonal matrix with $z_1, \dots, z_n$ as the diagonal entries.

For an interval hull $I(\mathcal{A}, \mathcal{B})$,  its center and its radius,  denoted by $\mathcal{I}_c$
and $\Delta$, respectively, are defined as follows:
\begin{center}
$\mathcal{I}_c = \frac{\mathcal{B} + \mathcal{A}}{2}$ and $\Delta = \frac{\mathcal{B}- \mathcal{A}}{2}.$
\end{center}

From the definition, it is clear that $\Delta \geq 0$.

\begin{defn}[\cite{ya-ya1}]
For an $m$-order $n$-dimensional tensor $\mathcal{A}$ and an $n \times n$ matrix $B$, the product $\mathcal{C} = \mathcal{A} \overbrace{B \cdots B}^{m} $ is defined as follows:
\begin{center}

$C_{i_1\dots i_m} = \sum_{j_1,\dots , j_m=1}^{n} A_{j_1 \dots j_m}  B_{i_1 j_1} \cdots B_{i_m j_m} $.

\end{center}
\end{defn}

For an interval hull of tensors $I(\mathcal{A}, \mathcal{B})$ and for each $z \in \mathbf{Z}$ we define $\mathcal{I}_z = \mathcal{I}_c - \Delta \overbrace{D_z \cdots D_z}^{m }$. It is easy to see that $\mathcal{I}_z \in I(\mathcal{A}, \mathcal{B})$ for all $z \in \mathbf{Z}$.

For a vector $x \in \mathbb{R}^n$, we define its sign vector $z = sgn(x)$ by $$ z_i =
\left\{
	    \begin{array}{ll}
		1 & \mbox{if } x_i \geq 0, \\
		-1 & \mbox{otherwise}.
	\end{array}
\right.$$

The following result is an extension of  \cite[Theorem 2.1]{rohn2} to  tensors.

\begin{theorem}\label{key_psd}
Let $\mathcal{A}$ and $\mathcal{B}$ be two $m$-order $n$-dimensional tensors and $\mathcal{A} \leq \mathcal{B}$. Assume that $x \in \mathbb{R}^n$ and  $z = sgn(x)$. Then for each $\mathcal{C} \in I(\mathcal{A}, \mathcal{B})$ and each $i \in \{1, \dots ,n \}$ we have $x_i (\mathcal{C}x^{m-1})_i \geq x_i(\mathcal{I}_z x^{m-1})_i.$

\end{theorem}

\begin{proof}

Let $\mathcal{C} \in I(\mathcal{A}, \mathcal{B})$ and $i \in \{1, \dots, n\}$. Then,
\begin{eqnarray}
 |x_i (\mathcal{C}x^{m-1})_i - x_i (\mathcal{I}_c x^{m-1})_i | &=&  |x_i ((\mathcal{C} - \mathcal{I}_c)x^{m-1})_i|     \nonumber \\
   &\leq& |x_i| |((\mathcal{C} - \mathcal{I}_c)x^{m-1})_i| \nonumber \\
   & \leq & |x_i| (\Delta |x|^{m-1})_i. \nonumber
\end{eqnarray}
Thus $x_i(\mathcal{C}x^{m-1})_i \geq   x_i (\mathcal{I}_c x^{m-1})_i - |x_i| (\Delta |x|^{m-1})_i $. Also we have $z = sgn(x)$ and $|x_i| = z_i x_i$ for all $i \in \{1, \dots ,n\}$. Now,
\begin{eqnarray}
 x_i (\mathcal{C}x^{m-1})_i & \geq & x_i (\mathcal{I}_c x^{m-1})_i - |x_i| (\Delta |x|^{m-1})_i \nonumber \\
      & = & \sum_{i_2, \dots, i_m = 1}^{n} x_i (\mathcal{I}_c)_{ii_2\dots i_m} x_{i_2} \cdots x_{i_m}- |x_i| \Delta_{ii_2 \dots i_m}|x|_{i_2}\cdots |x|_{i_m} \nonumber \\
      & = &  x_i \sum_{i_2, \dots, i_m = 1}^{n} ((\mathcal{I}_c)_{ii_2\dots i_m} -  \Delta_{ii_2 \dots i_m}z_i z_{i_2} \cdots z_{i_m})x_{i_2} \cdots x_{i_m} \nonumber \\
      & = &  x_i (\mathcal{I}_z x^{m-1})_i . \nonumber
\end{eqnarray}
Hence the proof.
\end{proof}

The following result is an extension of  \cite[Theorem 3.2]{rohn2} to tensors.

\begin{theorem}\label{int_p}
Let $\mathcal{A}$ and $\mathcal{B}$ be two $m$-order $n$-dimensional tensors and $\mathcal{A} \leq \mathcal{B}$. Then the following statements are equivalent:
\begin{itemize}

\item[(a)] $\mathcal{C}$ is  a $P$-tensor for all $\mathcal{C} \in I(\mathcal{A}, \mathcal{B})$.
\item[(b)] $\mathcal{I}_z$ is a $P$-tensor for all $z  \in \mathbf{Z}$.

\end{itemize}
\end{theorem}

\begin{proof}
$(a)\Rightarrow(b)$ is obvious. For the reverse, suppose $\mathcal{I}_z$ is a $P$-tensor for all $z \in \mathbf{Z}$. Let $\mathcal{C} \in  I(\mathcal{A}, \mathcal{B})$, $x \in \mathbb{R}^n$ and $z = sgn(x)$. Then by Theorem \ref{key_psd}, we have $x_i (\mathcal{C}x^{m-1})_i \geq x_i (\mathcal{I}_z x^{m-1})_i$ for all $i \in \{1, \dots, n\}$.
Hence $\mathcal{C}$ is a $P$-tensor.
\end{proof}

By the same proof:
\begin{theorem}\label{int_p_0}
Let $\mathcal{A}$ and $\mathcal{B}$ be two $m$-order $n$-dimensional tensors and $\mathcal{A} \leq \mathcal{B}$. Then the following statements are equivalent:
\begin{itemize}

\item[(a)] $\mathcal{C}$ is  a $P_0$-tensor for all $\mathcal{C} \in I(\mathcal{A}, \mathcal{B})$.
\item[(b)] $\mathcal{I}_z$ is a $P_0$-tensor for all $z  \in \mathbf{Z}$.

\end{itemize}
\end{theorem}

The following result is an extension of  \cite[Theorem 2]{rohn1} to tensors.

\begin{theorem}\label{int_psd}
Let $\mathcal{A}$ and $\mathcal{B}$ be two $m$-order $n$-dimensional tensors and $\mathcal{A} \leq \mathcal{B}$. Then the following statements are equivalent:
\begin{itemize}

\item[(a)] $\mathcal{C}$ is a positive semidefinite tensor for all $\mathcal{C} \in I(\mathcal{A}, \mathcal{B})$.
\item[(b)] $\mathcal{I}_z$ is a positive semidefinite tensor for all $z  \in \mathbf{Z}$.

\end{itemize}
\end{theorem}

\begin{proof}

$(a)\Rightarrow(b)$ is obvious. For the reverse, suppose $\mathcal{I}_z$ is a positive semidefinite tensor for all $z \in \mathbf{Z}$. Let $\mathcal{C} \in  I(\mathcal{A}, \mathcal{B})$, $x \in \mathbb{R}^n$ and $z = sgn(x)$. Then by Theorem \ref{key_psd}, we have $x_i (\mathcal{C}x^{m-1})_i \geq x_i (\mathcal{I}_z x^{m-1})_i$ for all $i \in \{1, \dots, n\}$.
Now, \begin{eqnarray}
\mathcal{C}x^m & = & \sum_{i_1, \dots, i_m =1}^{n} C_{i_1 \dots i_m} x_{i_1} \cdots x_{i_m} \nonumber \\
               & = &  \sum_{i_1 = 1 }^{n} x_{i_1} \sum_{i_2, \dots, i_m =1}^{n} C_{i_1 \dots i_m} x_{i_2} \cdots x_{i_m} \nonumber \\
               & \geq & \sum_{i_1 = 1}^{n} x_{i_1} (\mathcal{I}_z x^{m-1})_{i_1} \nonumber \\
               & = & \mathcal{I}_z x^{m} \geq 0. \nonumber
\end{eqnarray}

Thus $\mathcal{C}$ is a positive semidefinite tensor.
\end{proof}

The same proof shows:

\begin{theorem}\label{int_pd}
Let $\mathcal{A}$ and $\mathcal{B}$ be two $m$-order $n$-dimensional tensors and $\mathcal{A} \leq \mathcal{B}$. Then the following statements are equivalent:
\begin{itemize}

\item[(a)] $\mathcal{C}$ is a positive definite tensor for all $\mathcal{C} \in I(\mathcal{A}, \mathcal{B})$.
\item[(b)] $\mathcal{I}_z$ is a positive definite tensor for all $z  \in \mathbf{Z}$.

\end{itemize}
\end{theorem}

\textbf{Acknowledgement}
This work was supported by grant no. G-18-304.2/2011 by the German-Israeli Foundation for Scientific
Research and Development (GIF).

\bibliographystyle{siam}
\bibliography{final}

%
%
%
%

\end{document}